\documentclass[12pt]{article}
\usepackage{times, amsmath, amsfonts, amssymb, amsthm}
\usepackage{times}
\usepackage{srcltx}
\usepackage{setspace}

\newcommand{\dd}{\text{\rm d}}             

\newcommand{\1}{{\bf{1}}}

\theoremstyle{plain}
\newtheorem{theorem}{Theorem}

\newtheorem{proposition}[theorem]{Proposition}

\theoremstyle{definition}

\newtheorem{remark}[theorem]{Remark}

\renewenvironment{proof}[1][] {\noindent {\bf Proof#1.} }{\hspace*{\fill}$\square$\medskip\par}

\newcommand{\N}{{\mathbf N}}

\newcommand{\E}{{\mathbf E}}

\renewcommand{\P}{{\mathbf P}}

\author{Michael Scheutzow%
  \thanks{Institut f\"ur Mathematik, MA 7-5, Fakult\"at II, 
        Technische Universit\"at Berlin, 
        Stra\ss e des 17.~Juni 136, 10623 Berlin, FRG;  \ 
        \small{\tt ms{\scriptsize @}math.tu-berlin.de}}}

\title{A Stochastic Gronwall Lemma}

\date{\today}

\begin{document}  \maketitle

\begin{abstract}
%
\noindent We prove a stochastic Gronwall lemma of the following type: if $Z$ is an adapted nonnegative continuous process which satisfies a linear integral inequality 
with an added continuous local martingale $M$ and a process $H$ on the right hand side, then for any $p \in (0,1)$ the $p$-th moment of the supremum of $Z$ is 
bounded by a constant $\kappa_p$ (which does not depend on $M$) times the  $p$-th moment of the supremum of $H$. Our main tool is a martingale inequality 
which is due to D.~Burkholder. We provide an alternative simple proof of the martingale inequality which provides an explicit numerical value for the constant 
$c_p$ appearing in the inequality which is at most four times as large as the optimal constant. 

  \par\medskip

\end{abstract}


\vspace{.1cm}


In this note we first state a martingale inequality which is due to D.~Burkholder \cite{B75} and which estimates the $p$-th moment of the 
supremum of a continuous local martingale by a constant $c_{p}$ times the $p$-th moment of its negative infimum for $0<p<1$. Then we apply the martingale inequality and prove a 
stochastic Gronwall lemma for a nonnegative process $Z$.  The stochastic Gronwall lemma is useful when proving 
existence and uniqueness of solutions to stochastic differential equations satisfying only a one-sided Lipschitz condition (where the usual 
proof using the Burkholder-Davis-Gundy inequality does not apply). The point of the stochastic Gronwall lemma is that it provides an upper bound for the 
$p$-th moment of $Z$ which does not depend on the local martingale $M$ on the right-hand side of the inequality. The price one has to pay for this uniformity 
with respect to $M$ is that one has to assume $p<1$. 

We start by formulating the martingale inequality.
\begin{proposition}\label{martin} For each $p \in (0,1)$ and each continuous local martingale 
$M(t), \, t \ge 0$ starting at $M(0)=0$, we have
\begin{equation}\label{inequality}
\E \big(\sup_{t \ge 0} M^p(t)\big) \le c_p \E \big( (-\inf_{t \ge 0} M(t))^p\big),
\end{equation}
where $c_p:=\Big(4\wedge \frac 1p\Big) \frac{\pi p}{\sin( \pi p)}$.
\end{proposition}

\begin{remark} The inequality was first proved by D.\ Burkholder ({\cite{B75}, Theorem 1.4}) even a bit more generally (for a larger class of functions than the $p$-th power) 
but without an explicit estimate of the numerical value of $c_p$. We provide a short and elementary alternative proof below.
\end{remark}

\begin{remark} It is clear that the previous proposition does not extend to $p \ge 1$: consider the continuous martingale 
$M(t):=W(\tau_{-1}\wedge t)$ where $W$ is standard Brownian motion and $\tau_x:=\inf\{s \ge 0: W(s)=x\}$. Then the left hand side of \eqref{inequality} 
is infinite for each $p \ge 1$ while the right hand side is finite. This example also shows that even though the constant $c_p$ is certainly not optimal, 
it is at most off from the optimal constant by the factor $4\wedge (1/p)$ (which converges to one as $p$ approaches one). 
It is also clear that the proposition does not extend to right-continuous 
martingales: consider a martingale which is constant except for a single jump at time 1 of height 1 with probability $\delta$ and 
height $-\frac \delta{1-\delta}$ with probability $1-\delta$ where $\delta \in (0,1)$. It is straightforward to check that for an 
inequality of type $(\E \sup_{t \ge 0}M^p(t))^{1/p}\le c_{p,q} (\E (-\inf_{t \ge 0}M(t))^q)^{1/q}$  to hold for this class of examples for some finite $c_{p,q}$, 
we require that $q\ge 1$ irrespective of the value of $p \in (0,1)$.
\end{remark}

\begin{proof}[ of Proposition \ref{martin}] 
Since $M$ is a continuous local martingale starting at 0 it can be represented as a time-changed Brownian motion $W$ (on a suitable probability space). 
We can and will assume that $M$ converges almost surely (otherwise there is nothing to prove), 
so there exists an almost surely finite stopping time $T$ for $W$ such that 
$A:=\sup_{0 \le t \le T}W(t)=\sup_{0 \le t}M(t)$ and $B:=-\inf_{0 \le t \le T}W(t)=-\inf_{0 \le t}M(t)$. Let $0=a_0<a_1<...$ be a sequence which converges 
to $\infty$ and define
$$
\tau_i:=\inf\{t \ge 0:W(t)=-a_i\},\; Y_i:= \sup_{\tau_{i-1} \le t \le \tau_i} W(t),\,i \in \N, \quad N:=\inf\{i\in \N: \tau_i \ge T\}. 
$$  
The $Y_i$ are independent by the strong Markov property of $W$ and for $p \in (0,1)$ and $i \in \N$ we have
$$
\Gamma_i:=\E (Y_i\vee 0)^p= \frac{a_i-a_{i-1}}{a_i^{1-p}} \int_0^{\infty}\frac 1{1+y^{1/p}}\,\dd y = \frac{a_i-a_{i-1}}{a_i^{1-p}} \frac{\pi p}{\sin( \pi p)}. 
$$
Therefore, 
\begin{align*}
\E A^p &\le \sum_{n=1}^{\infty} \E \Big( \sup\{Y_1,...,Y_n\}^p\1_{N=n}\Big)\le \sum_{n=1}^{\infty}  \sum_{i=1}^n  \E  \Big((Y_i\vee 0)^p \1_{N= n}\Big)   \\
 &= \sum_{i=1}^{\infty} \E  \Big((Y_i\vee 0)^p \1_{N\ge i}\Big) = \sum_{i=1}^{\infty}  \Gamma_i \P\{N \ge i\},
\end{align*}
where the last equality again follows from the strong Markov property. Inserting the formula for $\Gamma_i$, choosing the particular values $a_i=c\gamma^i$ 
for some $c>0$ and $\gamma>1$, and observing that $\P\{N \ge i\} \le \P\{B \ge a_{i-1}\}$, we get
\begin{align*}
\E A^p &\le  \frac{\pi p}{\sin( \pi p)} c^p \Big( \gamma^p +\Big(1-\frac 1\gamma \Big) \sum_{i=2}^{\infty}\gamma^{ip}\P\{ B \ge c \gamma^{i-1}\}\Big)\\
&= \frac{\pi p}{\sin( \pi p)} c^p \Big( \gamma^p + \Big(1-\frac 1\gamma \Big)  \sum_{j=2}^{\infty}\P\{ B \in [c \gamma^{j-1},c\gamma^j)\} 
\Big(\frac{\gamma^{p(j+1)}-1}{\gamma^p-1} -1 - \gamma^p \Big)\Big)\\
&\le \frac{\pi p}{\sin( \pi p)}  \Big( c^p\gamma^p +\Big(1-\frac 1\gamma \Big) \frac{\gamma^{2p}}{\gamma^p-1} \E B^p-c^p\Big(1-\frac 1\gamma \Big) 
\Big( \frac 1{\gamma^p-1}+1+\gamma^p\Big)\P\{B \ge c\gamma\}\Big). 
\end{align*}
Dropping the last (negative) term, letting $c \to 0$ and observing that the function of $\gamma$ in front of $\E B^p$ converges to $1/p$ as $\gamma \to 1$ and 
that $\inf_{\gamma>1}\gamma^{2p}/(\gamma^p-1)=4$ we obtain the assertion.
\end{proof}

Next, we apply the martingale inequality to prove a stochastic Gronwall lemma. A similar stochastic Gronwall lemma was proved and used in 
\cite{RS10} in order to prove existence and uniqueness of a solution to a stochastic functional differential equation satisfying a one-sided 
Lipschitz condition only. That result was slightly more general in the sense that on the right hand side of equation \eqref{Gron} $Z$ was replaced by its 
running supremum, but it was less general concerning the function $\psi$ and it required higher moments of $H^*$. The proof did not explicitly 
use a martingale inequality.

For a real-valued process denote $Y^*(t):=\sup_{0 \le s \le t} Y(s)$.

\begin{theorem}
Let $c_p$ be as in Theorem \ref{martin}. Let $Z$ and $H$ be nonnegative, adapted processes with continuous paths and assume that $\psi$ is nonnegative and progressively measurable. 
Let $M$ be a continuous local martingale starting at 0. If 
\begin{equation}\label{Gron}
Z(t) \le \int_0^t \psi(s) Z(s)\,\dd s + M(t) +H(t)
\end{equation}
holds for all $t \ge 0$, then for $p\in(0,1)$, and $\mu,\nu>1$ such that $\frac 1\mu + \frac 1\nu =1$ and $p\nu<1$, we have 
\begin{equation}\label{eins}
\E \sup_{0 \le s \le t} Z^p(s) \le (c_{p\nu}+1)^{1/\nu}  \Big(\E\exp\{p\mu\int_0^t \psi(s)\,\dd s\}\Big)^{1/\mu}  \Big(\E (H^*(t))^{p\nu}\Big)^{1/{\nu}}.
\end{equation}
If $\psi$ is deterministic, then
\begin{equation}\label{zwei}
\E \sup_{0 \le s \le t}\ Z^p(s) \le (c_p+1) \exp\{p\int_0^t \psi(s)\,\dd s\}  \Big(\E (H^*(t))^p\Big),
\end{equation}
and
\begin{equation}\label{drei}
\E Z(t) \le \exp\{\int_0^t \psi(s)\,\dd s\} \E H^*(t).
\end{equation}
\end{theorem}

\begin{proof} Let $L(t):=\int_0^t \exp\{-\int_0^s \psi(u)\,\dd u\}\,\dd M(s)$. 
Applying the usual Gronwall Lemma (for each fixed $\omega \in \Omega$) to $Z$ and integrating by parts, we obtain
\begin{equation}\label{estim}
Z(t) \le \exp\{\int_0^t\psi(s)\,\dd s\}( L(t)+H^*(t)).
\end{equation}
Since $Z$ is nonnegative, we have $-L(t)\le H^*(t)$ for all $t \ge 0$. Therefore, using Proposition \ref{inequality} and 
H\"older's inequality, we get
\begin{align*}
\E(Z^*)^p(t) &\le \Big(\E \exp\{p\mu\int_0^t\psi(s)\,\dd s\}\Big)^{1/\mu}\Big( \E (L^*(t))^{p\nu} +\E (H^*(t))^{p\nu}\Big)^{1/\nu}\\ 
&\le \Big(\E \exp\{p\mu\int_0^t\psi(s)\,\dd s\}\Big)^{1/\mu} (  c_{p\nu}+1)^{1/\nu}  \Big(\E (H^*(t))^{p\nu}\Big)^{1/\nu}, 
\end{align*}
which is \eqref{eins}. Inequality \eqref{zwei} follows similarly.
The final statement follows by applying \eqref{estim} to $\tau_n \wedge t$ for a sequence of localizing stopping times 
$\tau_n$ for $L$ and applying Fatou's Lemma. 
\end{proof}

{\em Acknowledgement.} It is a pleasure to thank R.~Ba{\~n}uelos for drawing my attention to the cited paper of Donald Burkholder.


\begin{thebibliography}{10}%
\bibliographystyle{abbrv}

\bibitem{B75}
  Burkholder, D.\ L. (1975).  
   \newblock One-sided maximal functions and $H^p$,
   \newblock{\em J. Func. Anal.}  {\bf 18}, 429--454.     

\bibitem{RS10} 
  v. Renesse, M., Scheutzow, M. (2010).  
   \newblock Existence and uniqueness of solutions of stochastic functional differential equations,
   \newblock{\em Random Oper. Stoch. Equ.}  {\bf 18}, 267--284.     
\end{thebibliography}
\end{document}